\documentclass{article}

\usepackage{amsthm}
\usepackage{amsmath,latexsym,amsfonts,amssymb,graphicx}
\usepackage{graphicx}[dvipdfmx]
\usepackage{color}
\usepackage[margin=20mm]{geometry}

\newcommand{\R}{\mathbb{R}}

\def\vec#1{\mbox{\boldmath $#1$}}

\newtheorem{theorem}{Theorem}
\newtheorem{lemma}{Lemma}
\newtheorem{remark}{Remark}
\newtheorem{corollary}{Corollary}

\newtheorem{assumption}{Assumption}
\newtheorem{proposition}{Proposition}

\begin{document}

\title{Asymptotic analysis of the Gaussian kernel matrix\\ for partially noisy data in high dimensions}

\author{Kensuke Aishima
\thanks{Faculty of Computer and Information Sciences,  
  Hosei University, Tokyo 184-8585, Japan
  (\texttt{aishima@hosei.ac.jp}).}
}

\maketitle

\begin{abstract}
The Gaussian kernel is one of the most important kernels,
applicable to many research fields, including scientific computing and data science.
In this paper, we present asymptotic analysis
of the Gaussian kernel matrix
in high dimension under a statistical model of noisy data.
The main result is a nice combination of Karoui's asymptotic analysis
with procedures of constrained low rank matrix approximations.
More specifically, Karouli clarified an important asymptotic structure of the Gaussian kernel matrix,
leading to strong consistency of the eigenvectors,
though the eigenvalues are inconsistent.
This paper focuses on the above results and presents a consistent estimator with the use of the smallest eigenvalue,
whenever the target kernel matrix tends to low rank in the asymptotic regime.
Importantly, asymptotic analysis is given under a statistical model representing partial noise.
Although a naive estimator is inconsistent,
applying an optimization method for low rank approximations with constraints,
we overcome the difficulty caused by the inconsistency,
resulting in a new estimator with strong consistency in rank deficient cases.

\end{abstract}

\noindent
{\em Keywords}: Eigenvalue problems, Gaussian kernel, low rank approximation, consistent estimation

\noindent
{\em PACS}: 65F15, 65F55, 15A18, 62H12, 60B12  

\section{Introduction}
Kernel methods are standard techniques applicable to many fields of research, 
such as scientific computing and data science. The Gaussian kernel, in particular, 
is one of the most important kernels and is often applied in conjunction with 
support vector machines (SVM), principal component analysis (PCA), and so forth. 
This study addresses the asymptotic analysis of the Gaussian kernel.

The Gaussian kernel is mathematically described below.
Throughout the paper, for positive integers $p$ and $q$,
$I_{p}$ is a $p \times p$ identity matrix,
and $0_{p\times q}$ is a $p\times q$ zero matrix.
In addition, $\left[\cdot \right]_{pq}$ denotes
the $(p,q)$ elements of an argument matrix.
Regarding the norm, $\|\cdot \|$ means the 2-norm.
Given $n$ vectors
$\vec{x}_{i}\in \R^{d}\ (i=1,\ldots ,n)$,
define
\begin{align}\label{eq:gram}
\left[ K(\vec{x}_{1},\ldots ,\vec{x}_{n})\right]_{ij}
=
\exp \left( -\frac{\|\vec{x}_i-\vec{x}_j\|^2}{c_{d}} \right)
\qquad (i,j=1,\ldots ,n),
\end{align}
where $c_{d}$ is a scaling parameter.
The elements of the above matrix are presented as 
\begin{align*}
\begin{bmatrix}
   1 & \exp \left(-\frac{\|\vec{x}_{1}-\vec{x}_{2}\|^2}{c_d}\right) & \hdots & \exp \left(-\frac{\|\vec{x}_{1}-\vec{x}_{n}\|^2}{c_d}\right)  \\
   \exp \left(-\frac{\|\vec{x}_{2}-\vec{x}_{1}\|^2}{c_d}\right) & 1 & \ddots &  \vdots  \\
   \vdots & \ddots & \ddots&  \exp \left(-\frac{\|\vec{x}_{n-1}-\vec{x}_{n}\|^2}{c_d}\right)  \\
   \exp \left(-\frac{\|\vec{x}_{n}-\vec{x}_{1}\|^2}{c_d}\right) & \hdots & \exp \left(-\frac{\|\vec{x}_{n}-\vec{x}_{n-1}\|^2}{c_d}\right) &  1  
\end{bmatrix},
\end{align*}
where all the diagonal elements are 1.

The above matrix is interpreted as the Gram matrix
due to the kernel trick. 
Rigorously, the above matrix is the Gram matrix 
given by the inner product in
the reproducing kernel Hilbert space
as the feature space of nonlinear feature map.
With such a background, 
the asymptotic theory of kernel methods is often discussed in conjunction with 
the approximation theory of integral kernels~\cite{DP2025,KS2025IMA}.
In addition, under a statistic model, asymptotic structures of the Gaussian kernel matrices
are often analyzed together with some applications~\cite{NYA2020,NYA2021}.
For noisy data, the deconvolution kernel density estimator
is usually applied for efficiently removing random errors~\cite{Fan1991,LV1998}.
Regarding applied research, there are studies specializing in the Gaussian kernel on data containing missing values~\cite{SCJ2024,MGCJN2019}.
Among many researches, this study is closely relevant to
Karoui~\cite{Karoui2010} presenting an excellent pioneering work,
focusing on spectral analysis, 
specifically proving the consistency of eigenvectors for the Gaussian kernel matrix in the asymptotic regime.
On the basis of this, research into spectral distributions and their application to clustering 
has been actively pursued~\cite{AR2021,DM2023,Fan2019}.

The purpose of this study is to combine Karoui's asymptotic analysis~\cite{Karoui2010}
with procedures of constrained low rank matrix approximations.
Similarly to Karoui's formulation, the observation vectors $\vec{x}_{1},\ldots ,\vec{x}_{n} \in \R^{d}$ are contaminated with random noise.
Then, the eigenvectors of the kernel matrix
are consistent estimators due to its asymptotic matrix structure
under mild assumptions.
From such a perspective, we construct a consistent estimator
of the true kernel matrix in rank deficient cases.
The main contribution of this paper is
to develop the above asymptotic analysis and yield a specific estimator with strong consistency
for a more complicated statistical model
representing partial noise, where the existing naive estimations cannot preserve the consistency of the eigenvectors.

From another viewpoint,
the contribution addresses matrix eigenvalue analysis and low rank matrix approximations in numerical linear algebra.
The matrix low rank approximations
are obtained by the singular value decomposition (SVD)~\cite{GV2013,Saad2011},
successfully applied to constrained problems~\cite{Demmel1987,GHS1987}. 
Such important results for the low rank approximations are usually applied to the total least squares (TLS)~\cite{GV1980}.
Due to the importance of structured TLS by constraints, many efficient numerical methods have been actively developed~\cite{Hnetynkova2011,Huffel1987,Huffel1989,Huffel1991,LJY2022},
and the solvability conditions with the condition numbers are well analyzed~\cite{LJ2022,PW1993,wei1998,ZMW2017}.
In addition, the TLS for low rank coefficient matrices are well studied in~\cite{FB1994,MZW2020,MDB2021,Park2011}.
Importantly, the corresponding statistical asymptotic analysis is actively being developed~\cite{Aishima2024,Aishima2025,DK2018,KH2004,KMH2005}.
Moreover, note that low rank approximations of kernel matrices have recently received a lot of attention~\cite{AP2023,CETW2025,KS2025,WLD2018,WLMD2019}.
With such a background, for the Gaussian kernel matrix, we construct an estimator with strong consistency in rank deficient cases
based on the above constrained low rank matrix approximations. 
To the best of the author's knowledge,
this is the first study to directly apply the above procedures for the low rank matrix approximations with the constraints
to the asymptotic analysis of kernel methods.

This paper is organized as follows.
Section~\ref{sec:aim} is a description of
our problem setting, where we perform asymptotic analysis.
In Section~\ref{sec:mainresult}, the asymptotic analysis
for the Gaussian kernel matrix is presented,
where strong consistency of the eigenvectors is proved in the same manner as Karoui's result~\cite{Karoui2010},
newly constructing a consistent estimator of the kernel matrix in rank deficient cases.
Section~\ref{sec:mainresult2} presents the main results
of the asymptotic analysis under a statistical model representing partial noise,
resulting in a new estimator with strong consistency in rank deficient cases.
Finally, Section~\ref{sec:conclusion} gives the conclusion.

\section{Problem setting for consistency analysis in high dimension}\label{sec:aim}
With the research background in the previous section, 
we formulate a noise model
to construct an estimator. 
Basically, our formulation is a simplification of the formulation by Karoui~\cite{Karoui2010}.
In~\cite{Karoui2010}, general kernel matrices are discussed,
where the $(i,j)$ elements of the kernel matrix are given by $f({\|\vec{x}_{i}-\vec{x}_{j}\|_{2}}^{2})$
for $f$ in a certain class $\mathcal{F}$.
In such asymptotic analysis, since $\sup_{f \in \mathcal{F}} G(f)$ for objective functions $G$ is often considered,
the outer-probability statements are required.
Although it is not an issue for the rigorous analysis, 
for the readability, we reformulate the problem
restricted to the Gaussian kernel matrix,
resulting in a simple asymptotic analysis.

In the following, we consider random $d$ dimensional vectors 
$\vec{x}_{i}\in \R^{d}\ (i=1,\ldots ,n)$
given by
\begin{align}
\vec{x}_{i}=\vec{s}_{i}+\vec{\xi}_{i}\qquad (i=1,\ldots ,n),
\end{align}
where $\vec{s}_{i},\vec{\xi}_{i}\in \R^{d}\ (i=1,\ldots ,n)$ are 
important signal vectors and
random noisy vectors, respectively.
For the signal vectors $\vec{s}_{i}\in \R^{d}\ (i=1,\ldots ,n)$, we assume the following.

\begin{assumption}\label{as:s}
For integers $i,j\ (1\le i,j \le n)$, assume that $\lim_{d\to \infty}d^{-1} {\vec{s}_{i}}^{\top}\vec{s}_{j} $ exist, and
\begin{align}
\lim_{d\to \infty}\frac{\|\vec{s}_{i}\|^2}{d}>0 \qquad (i=1,\ldots ,n).
\end{align}
\end{assumption}

The above condition in Assumption~\ref{as:s}
is natural if we consider the case where the vectors $\vec{s}_{1}\ldots ,\vec{s}_{n}$
are obtained by discrete approximations of $L^2$ functions.
The purpose is to obtain $K(\vec{s}_{1},\ldots ,\vec{s}_{n})$
by some estimators with the use of $\vec{x}_{1},\ldots ,\vec{x}_{n}$,
and thus we present asymptotic analysis of the estimators below.

For statistical asymptotic analysis of random vectors, 
we explain some general statistical terms and concepts. 
In general, a sequence of random variables $\mathcal{X}^{(0)},\mathcal{X}^{(1)},\mathcal{X}^{(2)},\ldots $
converges to a random variable $\mathcal{X}$ with probability one
if and only if
\begin{align*}
\mathcal{P}(\{ \omega \in \Omega \mid \lim_{m\to \infty}\mathcal{X}^{(m)}(\omega)=\mathcal{X}(\omega) \})=1,
\end{align*}
where $\mathcal{P}$ is a probability measure on a sample space $\Omega$.
Strong convergence stands for the convergence with probability one.
In this paper, the random variable argument $\omega$ is omitted for simplicity.
Let $\widehat{\mathcal{X}}^{(m)}$ denote 
an estimator of a target constant value $\theta$
with the use of a sample of size $m$.
If $\widehat{\mathcal{X}}^{(m)}$
converges to $\theta$ with probability one,
$\widehat{\mathcal{X}}^{(m)}$
is said to be an estimator with strong consistency.

The following asymptotic analysis discusses the limit of the dimension $d\to \infty$,
where the elements of random vectors are interpreted as the corresponding sequence.
Note that the squared 2-norm of a vector in $\R^{d}$ is the sum of the squares of its $d$ elements. 
In the following, 
we analyze the consistency concerning the Gaussian kernel
under the following assumption.

\begin{assumption}\label{as:xi}
Assume that
\begin{align}\label{eq:keyxi}
& \lim_{d\to \infty}\frac{\| \vec{\xi}_{i} \|^2}{d}=\bar{\sigma}^2 \qquad (1\le i \le n),
\\
\label{eq:keyxi2}
& \lim_{d\to \infty}\frac{{\vec{\xi}_{i}}^{\top}\vec{\xi}_{j} }{d}=0  \qquad (1\le i \not = j \le n),
\\
\label{eq:keyxis}
& \lim_{d\to \infty}\frac{{\vec{s}_{i}}^{\top}\vec{\xi}_{j} }{d}=0  \qquad (1\le i , j \le n)
\end{align}
with probability one.
\end{assumption}

Concerning the strong convergence as in the above assumption,
usually assume that all $n$ mean vectors of $\vec{\xi}_{1},\ldots ,\vec{\xi}_{n}$
are $\vec{0} \in \R^{d}$. 
In addition, 
if all the elements of all vectors $\vec{\xi}_{1},\ldots ,\vec{\xi}_{n}$ are independent and identically distributed (i.i.d.)
and the variances are bounded, we have \eqref{eq:keyxi}, \eqref{eq:keyxi2}, and \eqref{eq:keyxis}
with probability one due to the strong law of large numbers (SLLN).
In fact, regarding the i.i.d., 
the identical distributions are not necessary for the convergence presented in Assumption~\ref{as:xi}.

\begin{proposition}
Assume that $\vec{\xi}_{1},\ldots ,\vec{\xi}_{n}$ are independent and identically distributed (i.i.d.) and all $n$ mean vectors are $\vec{0} \in \R^{d}$. 
Let ${\sigma_{1}}^{2},\ldots ,{\sigma_{d}}^{2}$ denote the second moments of the corresponding elements of the random vectors,
and assume that ${\sigma_{1}}^{2},\ldots ,{\sigma_{d}}^{2}$ are uniformly bounded as $d \to \infty$. In addition, assume that $\lim_{d\to \infty}\sum_{i=1}^{d}d^{-1}{\sigma_{i}}^{2}$ exists, and let $\bar{\sigma}^{2}=\lim_{d\to \infty}\sum_{i=1}^{d}d^{-1}{\sigma_{i}}^{2}$.
Moreover, the forth moments of all the elements of the random vectors are uniformly bounded as $d \to \infty$. 
Then, \eqref{eq:keyxi}, \eqref{eq:keyxi2}, and \eqref{eq:keyxis} hold with probability one.
\end{proposition}
\begin{proof}
If the variances are uniformly bounded, the SLLN is established by Kolmogorov's theory; see~\cite[\S 2]{Durrett} for the details.
Thus, we have \eqref{eq:keyxi} and \eqref{eq:keyxi2} with probability one. 
Regarding \eqref{eq:keyxis}, we use Abel's partial summation formula.
For any $i,j$, define $s_{i,1},\ldots , s_{i,d}$ and $\xi_{j,1},\ldots , \xi_{j,d}$ as
the elements of $\vec{s}_{i}$ and $\vec{\xi}_{j}$, respectively. 
In addition, define 
\begin{align*}
t_{i,0}=0,\quad t_{i,d'}=\sum_{k=1}^{d'}{s_{i,k}}^{2} \quad (d'=1,\ldots ,d),
\end{align*}
leading to
\begin{align}\label{eq:beta}
\lim_{d\to \infty} d^{-1}t_{i,d}<\infty 
\end{align}
from the assumption that $d^{-1}\| \vec{s}_{i}\|^2$ are bounded.
Moreover, note
\begin{align*}
\sum_{k=1}^{d}k^{-2}{s_{i,k}}^{2}=
\sum_{k=1}^{d}k^{-2}(t_{i,k}-t_{i,k-1})=
\sum_{k=1}^{d}k^{-2}t_{i,k}-\sum_{k=0}^{d-1}(k+1)^{-2}t_{i,k}=
d^{-2}t_{i,d}+\sum_{k=1}^{d-1}(k^{-2}-(k+1)^{-2})t_{i,k}.
\end{align*}
The second term on the right-hand side of the above equation is
$\sum_{k=1}^{d-1}(k^{-1}(k+1)^{-2}(2k+1))k^{-1}t_{i,k}$.
Then,
\eqref{eq:beta} and 
$\sum_{k=1}^{d-1}k^{-1}(k+1)^{-2}(2k+1)<\infty$
result in $\sum_{k=1}^{d}k^{-2}{s_{i,k}}^{2}<\infty$.
Thus, from the assumption that ${\sigma_{1}}^2,\ldots ,{\sigma_{d}}^{2}$ are uniformly bounded, 
we obtain
\begin{align*}
\lim_{d \to \infty}\sum_{k=1}^{d}k^{-2}\mathbb{E}((s_{i,k}\xi_{j,k})^{2})=
\lim_{d \to \infty}\left(\sup_{1\le k \le d}{\sigma_k}^2\right)\sum_{k=1}^{d}k^{-2}{s_{i,k}}^{2}<\infty,
\end{align*}
where $\mathbb{E}$ means the expectation value.
Noting \cite[eq. (7) in Corollary to Theorem~5.4.1]{Chung2001}, we obtain
\begin{align*}
\lim_{d \to \infty}\frac{\sum_{k=1}^{d}s_{i,k}\xi_{j,k}}{d}=0
\end{align*}
with probability one, implying \eqref{eq:keyxis} with probability one.
\end{proof}

Strong convergence in Assumption~\ref{as:xi} might be established
under more weak conditions according to recent asymptotic analysis for the SLLN.
Since this paper focuses on the asymptotic analysis of the Gaussian kernel, 
we will not further examine the SLLN for sequences of random variables 
and instead discuss convergence under Assumption~\ref{as:xi}.

\section{Estimator with strong consistency for the Gaussian kernel matrix for high dimensional noisy data}\label{sec:mainresult}

In this section, we present asymptotic analysis as $d\to \infty$
for the Gaussian kernel.
More specifically, we prove the consistency of eigenvectors 
of the Gaussian kernel matrix under Assumptions~\ref{as:s} and \ref{as:xi}.
Our asymptotic analysis shows that the kernel matrix and its eigenvalues are inconsistent.
The asymptotic analysis above is essentially a reproduction of~\cite[\S 2.3.2]{Karoui2010}. 
However, it focuses specifically on the analysis of the Gaussian kernel, 
rephrasing the conditions clearly in the previous section to facilitate subsequent analysis 
and clarifying the necessary points for subsequent analysis.
More specifically, on the basis of the results, we propose a modified kernel matrix
that is a consistent estimator of the true kernel matrix for rank deficient cases.
The next section applies this consistency analysis to a more general noise model. 
Thus, this analysis is considered an important fundamental analysis.

For the asymptotic analysis, recall that the Gaussian kernel matrix $K(\vec{x}_{1},\ldots ,\vec{x}_{n})$ is defined as in~\eqref{eq:gram}.
For the scaling parameter $c_{d}$, we assume the following.

\begin{assumption}\label{as:c}
$\lim_{d\to \infty}d^{-1}c_{d}$ exists, and let
\begin{align}\label{eq:gamma}
\lim_{d\to \infty}\frac{c_{d}}{d}=\gamma>0.
\end{align}
\end{assumption}

Note that, since the kernel matrix $K(\vec{x}_{1},\ldots ,\vec{x}_{n})$ is given by \eqref{eq:gram}
for $n$ sample vectors $\vec{x}_{1},\ldots ,\vec{x}_{n} \in \R^{d}$ under Assumptions~\ref{as:s} and \ref{as:xi},
Assumption~\ref{as:c} is reasonable for normalizing $K(\vec{x}_{1},\ldots ,\vec{x}_{n})$.

\subsection{Asymptotic analysis: strong consistency of invariant subspace}
Here we prove strong consistency of eigenvectors of 
$K(\vec{x}_{1},\ldots ,\vec{x}_{n})$ in \eqref{eq:gram} as $d\to \infty$.
The next lemma is key for understanding the convergence properties of the kernel matrices
$K(\vec{x}_{1},\ldots ,\vec{x}_{n})$ and $K(\vec{s}_{1},\ldots ,\vec{s}_{n})$.

\begin{lemma}
Under Assumptions~\ref{as:s} and \ref{as:c}, 
$\lim_{d\to \infty}K(\vec{s}_{1},\ldots ,\vec{s}_{n})$
exists.
Let $i,j$ denote integers such that $1\le i,j \le n$ and $i\not = j$.
Then, under Assumption~\ref{as:xi}, we have 
\begin{align}\label{eq:kxilim}
\lim_{d\to \infty}\left[ K(\vec{\xi}_{1},\ldots ,\vec{\xi}_{n})\right]_{ij}
=\exp\left(-2\gamma^{-1}\bar{\sigma}^{2}\right)
\quad \text{with probability one}.
\end{align}
In addition, we have
\begin{align}\label{eq:kxij}
\lim_{d\to \infty}\left[ K(\vec{x}_{1},\ldots ,\vec{x}_{n})\right]_{ij}
=
\exp\left(-2\gamma^{-1}\bar{\sigma}^{2}\right)
\lim_{d\to \infty}\left[ K(\vec{s}_{1},\ldots ,\vec{s}_{n})\right]_{ij}
\quad \text{with probability one}.
\end{align}
As the matrix representation,
\begin{align}\label{eq:hadamard}
\lim_{d\to \infty} K(\vec{x}_{1},\ldots ,\vec{x}_{n})
=
\lim_{d\to \infty} K(\vec{s}_{1},\ldots ,\vec{s}_{n})\odot
\lim_{d\to \infty}K(\vec{\xi}_{1},\ldots ,\vec{\xi}_{n})
\quad \text{with probability one},
\end{align}
where $\odot$ stands for the Hadamard product of two matrices.
\end{lemma}
\begin{proof}
First of all, it is clear that  
$\lim_{d\to \infty}K(\vec{s}_{1},\ldots ,\vec{s}_{n})$
exists under Assumptions~\ref{as:s} and \ref{as:c}.

Next, for \eqref{eq:kxilim},
we see
\begin{align}
\left[ K(\vec{\xi}_{1},\ldots ,\vec{\xi}_{n})\right]_{ij}
=\exp\left(-\frac{d}{c_{d}}\frac{\|\vec{\xi}_{i}-\vec{\xi}_{j}\|^2}{d}\right)
=\exp\left(-\frac{d}{c_{d}}\frac{\|\vec{\xi}_{i}\|^2+\|\vec{\xi}_{j}\|^2+{\vec{\xi}_{i}}^{\top}\vec{\xi}_{j}}{d}\right).
\end{align}
Thus, under Assumptions~\ref{as:xi} and \ref{as:c},
\eqref{eq:kxilim} follows from~\eqref{eq:keyxi}, \eqref{eq:keyxi2}, and \eqref{eq:gamma}.

For \eqref{eq:kxij}, we have
\begin{align*}
\exp\left(-\frac{\|\vec{x}_{i}-\vec{x}_{j}\|^2}{c_d}\right)
&=
\exp\left(-\frac{\|\vec{s}_{i}-\vec{s}_{j}+\vec{\xi}_{i}-\vec{\xi}_{j}\|^2}{c_d}\right)
\\
&=
\exp\left(-\frac{\|\vec{s}_{i}-\vec{s}_{j}\|^2+2(\vec{s}_{i}-\vec{s}_{j})^{\top}(\vec{\xi}_{i}-\vec{\xi}_{j})+\|\vec{\xi}_{i}-\vec{\xi}_{j}\|^2}{c_d}\right)
\\
&=
\exp\left(-\frac{\|\vec{s}_{i}-\vec{s}_{j}\|^2+2(\vec{s}_{i}-\vec{s}_{j})^{\top}(\vec{\xi}_{i}-\vec{\xi}_{j})+\|\vec{\xi}_{i}\|^2-2{\vec{\xi}_{i}}^{\top}\vec{\xi}_{j}+\|\vec{\xi}_{j}\|^2}{c_d}\right).
\end{align*}
Note that $\lim_{d\to \infty}d^{-1}\|\vec{s}_{i}-\vec{s}_{j}\|^2$ exists under Assumption~\ref{as:s}.
Thus, under Assumptions~\ref{as:xi} and \ref{as:c}, we have
\begin{align*}
\lim_{d\to \infty}\exp\left(-\frac{\|\vec{x}_{i}-\vec{x}_{j}\|^2}{c_d}\right)
=
\exp\left(-2\gamma^{-1}\bar{\sigma}^{2}\right)
\lim_{d\to \infty}\exp\left(-\frac{\|\vec{s}_{i}-\vec{s}_{j}\|^2}{c_d}\right).
\end{align*}
Thus, we obtain~\eqref{eq:kxij}.

For the matrix representation in~\eqref{eq:hadamard},
noting that the diagonal elements of the Gaussian kernel matrix are 1,
we see~\eqref{eq:hadamard} from \eqref{eq:kxilim} and \eqref{eq:kxij}.
\end{proof}

The following representation of the convergence of the kernel matrix in the next theorem is important.

\begin{theorem}\label{thm:main}
Under Assumptions~\ref{as:s} and \ref{as:c}, 
$\lim_{d\to \infty}K(\vec{s}_{1},\ldots ,\vec{s}_{n})$
exists.
In addition, under Assumption~\ref{as:xi}, we have 
\begin{align}\label{eq:keykx}
\lim_{d\to \infty} K(\vec{x}_{1},\ldots ,\vec{x}_{n})
=
\exp\left(-2\gamma^{-1}\bar{\sigma}^{2}\right)
(\lim_{d\to \infty} K(\vec{s}_{1},\ldots ,\vec{s}_{n})-I_{n})+I_{n}
\quad \text{with probability one}.
\end{align}
\end{theorem}
\begin{proof}
Noting \eqref{eq:kxij} and $\left[K(\vec{x}_{1},\ldots ,\vec{x}_{n})\right]_{ii}=1\ (i=1,\ldots ,n)$,
we obtain \eqref{eq:keykx}.
\end{proof}

The above theorem clarifies important features of eigenpairs of the kernel matrices.
To explain details, let
\begin{align}
\lambda_{1}\le \lambda_{2}\le \cdots \le \lambda_{n}
\end{align}
denote the eigenvalues of $K(\vec{x}_{1},\ldots ,\vec{x}_{n})$.
Similarly, let
\begin{align}\label{eq:mu}
\mu_{1}\le \mu_{2}\le \cdots \le \mu_{n}
\end{align}
denote the eigenvalues of $K(\vec{s}_{1},\ldots ,\vec{s}_{n})$.
Then, \eqref{eq:keykx} in Theorem~\ref{thm:main} shows the following.
\begin{corollary}\label{cor:eig}
Under the same assumption as in Theorem~\ref{thm:main}, we have
\begin{align}
\lim_{d\to \infty} \lambda_{i}=\exp(-2\gamma ^{-1}\bar{\sigma}^{2})\lim_{d\to \infty} \mu_{i}-\exp(-2\gamma ^{-1}\bar{\sigma}^{2})+1  
\quad (i=1,\ldots ,n)
\quad \text{with probability one}.
\end{align}
In addition, the eigenvectors of $K(\vec{x}_{1},\ldots ,\vec{x}_{n})$
are the consistent estimators of the eigenvectors of
$\lim_{d\to \infty}K(\vec{s}_{1},\ldots ,\vec{s}_{n})$.
\end{corollary}

The meaning of the consistency of the eigenvectors is described below.
Eigenvectors can have arbitrary magnitudes, and when eigenvalues are multiple, 
they correspond to any vector belonging to the corresponding subspace. 
Since we are considering symmetric matrices, multiple eigenvalues correspond to 
eigenspaces of dimension equal to the multiplicity. Therefore, due to the continuity of eigenvectors, 
any eigenvector of $K(\vec{x}_{1},\ldots ,\vec{x}_{n})$ is equal to 
one of the eigenvectors of $K(\vec{s}_{1},\ldots ,\vec{s}_{n})$
in the asymptotic regime as $d\to \infty$.

\subsection{Modified estimator for reconstructions of rank deficient matrices}\label{sec:modifiedestimator}
In the following, we will move beyond Karoui's results~\cite[\S 2.3.2]{Karoui2010} and enter into the fundamentally new content of this paper.
On the basis of the asymptotic analysis above,
we propose a modified estimator for reconstructions of $\lim_{d\to \infty}K(\vec{s}_{1},\ldots ,\vec{s}_{n})$
in rank deficient cases.
Thus, we assume the following.

\begin{assumption}\label{as:lowrank}
$\lim_{d\to \infty} K(\vec{s}_{1},\ldots ,\vec{s}_{n})$ is a low rank matrix.
\end{assumption}

Before asymptotic analysis, we explain a situation where $\lim_{d\to \infty} K(\vec{s}_{1},\ldots ,\vec{s}_{n})$ is a low rank matrix.
For example, if we assume
\begin{align}\label{eq:sij0}
\lim_{d\to \infty}\frac{\|\vec{s}_{i}-\vec{s}_{j}\|^2}{d}=0
\end{align}
for some fixed $i$ and $j$,
then the corresponding two rows and columns of $\lim_{d\to \infty} K(\vec{s}_{1},\ldots ,\vec{s}_{n})$
are the same ones, implying that $\lim_{d\to \infty} K(\vec{s}_{1},\ldots ,\vec{s}_{n})$ is low rank.
To explain in detail, for any integer $k$ such that $1\le k \le n$,
we see
\begin{align*}
\lim_{d\to \infty}\frac{\|\vec{s}_{i}-\vec{s}_{k}\|^2}{d}
=
\lim_{d\to \infty}\frac{\|\vec{s}_{i}-\vec{s}_{k}-\vec{s}_{j}+\vec{s}_{j}\|^2}{d}
=
\lim_{d\to \infty}\frac{\|\vec{s}_{j}-\vec{s}_{k}\|^2+\|\vec{s}_{i}-\vec{s}_{j}\|^2+2(\vec{s}_{j}-\vec{s}_{k})^{\top}(\vec{s}_{i}-\vec{s}_{j})}{d}
=
\lim_{d\to \infty}\frac{\|\vec{s}_{j}-\vec{s}_{k}\|^2}{d},
\end{align*}
where the last equality is due to
\begin{align*}
\lim_{d\to \infty}\left|\frac{(\vec{s}_{j}-\vec{s}_{k})^{\top}(\vec{s}_{i}-\vec{s}_{j})}{d}\right|
\le
\lim_{d\to \infty}\sqrt{\frac{\|\vec{s}_{j}-\vec{s}_{k}\|^2}{d}}\sqrt{\frac{\|\vec{s}_{i}-\vec{s}_{j}\|^2}{d}}=0
\end{align*}
from the Cauchy--Schwarz inequality.
Thus, the two rows and columns of $\lim_{d\to \infty} K(\vec{s}_{1},\ldots ,\vec{s}_{n})$ are the same ones.
Below, we consider the rank deficient case in the asymptotic regime.

Since the eigenvalues of $K(\vec{s}_{1},\ldots ,\vec{s}_{n})$
are defined as in \eqref{eq:mu},
the smallest eigenvalue has the limit value
\begin{align}\label{eq:mu0}
\lim_{d\to \infty} \mu_{1}=0
\end{align}
under Assumption~\ref{as:lowrank}.
Then, under the assumption in~\eqref{eq:mu0}, Corollary~\ref{cor:eig} shows
\begin{align}\label{eq:limlambda}
\lim_{d\to \infty} \lambda_{1}=
\exp(-2\gamma ^{-1}\bar{\sigma}^{2})\lim_{d\to \infty} \mu_{1}-\exp(-2\gamma ^{-1}\bar{\sigma}^{2})+1=-\exp(-2\gamma ^{-1}\bar{\sigma}^{2})+1,
\end{align}
which is the smallest eigenvalue of $\lim_{d\to \infty} K(\vec{x}_{1},\ldots ,\vec{x}_{n})$.

Noting the convergence feature above, 
we construct $\widetilde{K}(\vec{x}_{1},\ldots ,\vec{x}_{n})$ defined as follows.
For the diagonal elements, $[\widetilde{K}(\vec{x}_{1},\ldots ,\vec{x}_{n})]_{ii}=[K(\vec{x}_{1},\ldots ,\vec{x}_{n})]_{ii}(=1)\ (i=1,\ldots , n)$.
In addition, let 
\begin{align*}
\text{for $i\not = j$,}
\quad 
\left[ \widetilde{K}(\vec{x}_{1},\ldots ,\vec{x}_{n})\right]_{ij}=
(1-\lambda_{1})^{-1}\left[ K(\vec{x}_{1},\ldots ,\vec{x}_{n})\right]_{ij}.
\end{align*}
In other words, $\widetilde{K}(\vec{x}_{1},\ldots ,\vec{x}_{n})$ is defined as
\begin{align}\label{eq:tildeK}
\widetilde{K}(\vec{x}_{1},\ldots ,\vec{x}_{n})=
(1-\lambda_{1})^{-1}(K(\vec{x}_{1},\ldots ,\vec{x}_{n})-I_{n})+I_{n}.
\end{align}

The procedure for calculating the proposed estimator $\widetilde{K}(\vec{x}_{1},\ldots ,\vec{x}_{n})$ is summarized as follows.

\vspace{10pt}

\noindent
[Proposed estimator]
\begin{itemize}
\item Compute the Gram matrix of the Gaussian kernel ${K}(\vec{x}_{1},\ldots ,\vec{x}_{n})$ as in \eqref{eq:gram}

\item Compute the smallest eigenvalue $\lambda_{1}$ of ${K}(\vec{x}_{1},\ldots ,\vec{x}_{n})$

\item Compute $\widetilde{K}(\vec{x}_{1},\ldots ,\vec{x}_{n})$ in \eqref{eq:tildeK} as the estimator
\end{itemize}
Then, $\widetilde{K}(\vec{x}_{1},\ldots ,\vec{x}_{n})$ has strong consistency
as follows.

\begin{theorem}
Under Assumptions~\ref{as:s} and \ref{as:c}, 
$\lim_{d\to \infty}K(\vec{s}_{1},\ldots ,\vec{s}_{n})$
exists.
In addition, under Assumptions~\ref{as:xi} and \ref{as:lowrank}, we have 
\begin{align}\label{eq:litildeK}
\lim_{d\to \infty} \widetilde{K}(\vec{x}_{1},\ldots ,\vec{x}_{n})
=
\lim_{d\to \infty} K(\vec{s}_{1},\ldots ,\vec{s}_{n})
\quad \text{with probability one},
\end{align}
where $\widetilde{K}(\vec{x}_{1},\ldots ,\vec{x}_{n})$ is defined as in \eqref{eq:tildeK}.
\end{theorem}
\begin{proof}
From~\eqref{eq:keykx} in Theorem~\ref{thm:main},
\eqref{eq:litildeK} is an easy reduction
from~\eqref{eq:limlambda} and \eqref{eq:tildeK}.
\end{proof}

\begin{remark}\label{rem:rankdeficient}
Since the case of $\lim_{d\to \infty}d^{-1}\|\vec{s}_{i}-\vec{s}_{j}\|=0$ in \eqref{eq:sij0} is extremely exceptional, 
one might question the need to discuss the rank deficient case.
One reason for considering rank reduction scenarios is that methods approximating kernels 
with low rank have gained attention in recent years. Therefore, if a general estimation method exists 
for exact rank reduction cases, it could potentially be applied to error analysis in low rank approximations.
Another reason is that, if the asymptotic case as $n \to \infty$, 
the kernel matrix could be nearly low rank,
where the smallest singular values tend to 0,
though we do not consider such asymptotic case in this paper.
This paper discusses estimators for fixed $n$. The next section provides a consistent estimator 
for a statistical model representing partial noise. 
\end{remark}

\section{Proposed consistent estimator with asymptotic analysis for a model representing partial noise}\label{sec:mainresult2}
In this section, we discuss an extended noise model described below.
For $1\le \ell \le n-1$, let
\begin{align}\label{eq:xsxil}
\vec{x}_{i}=\vec{s}_{i} \quad (i=1,\ldots ,\ell),\qquad
\vec{x}_{i}=\vec{s}_{i}+\vec{\xi}_{i}\quad (i=\ell +1,\ldots ,n).
\end{align}
The following is an updated version of Assumption~\ref{as:xi} that reflects the noise model described above.
\begin{assumption}\label{as:xil}
Assume that
\begin{align}\label{eq:keyxil}
& \lim_{d\to \infty}\frac{\| \vec{\xi}_{i} \|^2}{d}=\bar{\sigma}^2 \qquad (\ell +1\le i \le n),
\\
\label{eq:keyxil2}
& \lim_{d\to \infty}\frac{{\vec{\xi}_{i}}^{\top}\vec{\xi}_{j} }{d}=0  \qquad (\ell +1\le i \not = j \le n),
\\
\label{eq:keyxisl}
& \lim_{d\to \infty}\frac{{\vec{s}_{i}}^{\top}\vec{\xi}_{j} }{d}=0  \qquad (1\le i \le n, \ell +1\le j \le n)
\end{align}
with probability one.
\end{assumption}

The purpose is to estimate ${K}(\vec{s}_{1},\ldots ,\vec{s}_{n})$
with the use of $\vec{x}_{1},\ldots ,\vec{x}_{n}$ in the asymptotic regime as $d\to \infty$,
whenever $\lim_{d\to \infty}{K}(\vec{s}_{1},\ldots ,\vec{s}_{n})$ is a low rank matrix.
To this end, we analyze an optimization method
for low rank approximations in~\cite{Demmel1987}.
Although numerous studies have followed~\cite{Demmel1987}, the purpose of this paper is 
to clarify the underlying principles. Therefore, we focus on the most fundamental optimization method, 
considering more detailed investigations as future work.

\subsection{Existing optimization method based on the Gaussian elimination for low rank approximations}
For constructing a consistent estimator,
we adopt an optimization method for a low rank matrix approximation.
For this purpose, let ${K}(\vec{x}_{1},\ldots ,\vec{x}_{n})$ be divided into
\begin{align}\label{eq:ksubx}
&{K}(\vec{x}_{1},\ldots ,\vec{x}_{n})=
\begin{bmatrix}
   {K}_{11}(\vec{x}_{1},\ldots ,\vec{x}_{n}) & {K}_{12}(\vec{x}_{1},\ldots ,\vec{x}_{n})  \\
   {K}_{21}(\vec{x}_{1},\ldots ,\vec{x}_{n}) & {K}_{22}(\vec{x}_{1},\ldots ,\vec{x}_{n})
\end{bmatrix}, \\
\label{eq:ksubx4}
&{K}_{11}(\vec{x}_{1},\ldots ,\vec{x}_{n}) \in \R^{\ell \times \ell},
\quad {K}_{12}(\vec{x}_{1},\ldots ,\vec{x}_{n}) \in \R^{\ell \times m},
\quad {K}_{21}(\vec{x}_{1},\ldots ,\vec{x}_{n}) \in \R^{m \times \ell},
\quad {K}_{22}(\vec{x}_{1},\ldots ,\vec{x}_{n}) \in \R^{m \times m},
\end{align}
where $\ell + m =n$.
Note that
\begin{align*}
{K}(\vec{x}_{1},\ldots ,\vec{x}_{n})={K}(\vec{s}_{1},\ldots ,\vec{s}_{\ell},\vec{x}_{\ell+1},\ldots ,\vec{x}_{n})
\end{align*}
from the noise model in~\eqref{eq:xsxil},
and ${K}_{11}(\vec{x}_{1},\ldots ,\vec{x}_{n}) \in \R^{\ell \times \ell}$ is a noiseless matrix
given by $\vec{s}_{1},\ldots ,\vec{s}_{\ell}$.
Noting this matrix structure,
we consider the following optimization problem:
\begin{align}\label{eq:optp}
\min_{\Delta \in \R^{m \times m}} \|\Delta \|_{\rm F}
\quad 
\text{such that}
\quad 
{\rm rank}\left(
\begin{bmatrix}
   {K}_{11}(\vec{x}_{1},\ldots ,\vec{x}_{n}) & {K}_{12}(\vec{x}_{1},\ldots ,\vec{x}_{n})  \\
   {K}_{21}(\vec{x}_{1},\ldots ,\vec{x}_{n}) & {K}_{22}(\vec{x}_{1},\ldots ,\vec{x}_{n})+\Delta
\end{bmatrix}
\right)
\le n-1,
\end{align}
where $\|\cdot \|_{\rm F}$ means the Frobenius norm.
For solving the above optimization problem, 
let us assume that ${K}_{11}(\vec{x}_{1},\ldots ,\vec{x}_{n})$ is nonsingular.
Then, similarly to the Gaussian elimination, 
\begin{align*}
&\begin{bmatrix}
   {K}_{11}(\vec{x}_{1},\ldots ,\vec{x}_{n}) & {K}_{12}(\vec{x}_{1},\ldots ,\vec{x}_{n})  \\
   {K}_{21}(\vec{x}_{1},\ldots ,\vec{x}_{n}) & {K}_{22}(\vec{x}_{1},\ldots ,\vec{x}_{n})
\end{bmatrix}
-
\begin{bmatrix}
   {K}_{11}(\vec{x}_{1},\ldots ,\vec{x}_{n})   \\
   {K}_{21}(\vec{x}_{1},\ldots ,\vec{x}_{n}) 
\end{bmatrix}
\begin{bmatrix}
I_{\ell} & {{K}_{11}(\vec{x}_{1},\ldots ,\vec{x}_{n})}^{-1}{K}_{12}(\vec{x}_{1},\ldots ,\vec{x}_{n})
\end{bmatrix}
\\
&=
\begin{bmatrix}
   0_{\ell \times \ell} & 0_{\ell \times m}  \\
   0_{m \times \ell} & {K}_{22}(\vec{x}_{1},\ldots ,\vec{x}_{n})-{K}_{21}(\vec{x}_{1},\ldots ,\vec{x}_{n}){{K}_{11}(\vec{x}_{1},\ldots ,\vec{x}_{n})}^{-1}{K}_{12}(\vec{x}_{1},\ldots ,\vec{x}_{n})
\end{bmatrix}.
\end{align*}
Thus, the condition in~\eqref{eq:optp} is equivalent to
\begin{align}\label{eq:optpsmall}
{\rm rank}\left(
{K}_{22}(\vec{x}_{1},\ldots ,\vec{x}_{n})-
{{K}_{21}(\vec{x}_{1},\ldots ,\vec{x}_{n})}{{K}_{11}(\vec{x}_{1},\ldots ,\vec{x}_{n})}^{-1}{K}_{12}(\vec{x}_{1},\ldots ,\vec{x}_{n})+\Delta
\right)
\le m-1.
\end{align}
The above low rank approximation problem is solved by the singular value decomposition (SVD)
based on the Eckart-Young-Mirsky theorem.
Since the above matrix is a symmetric matrix,
the eigenvalue with the smallest absolute value
is the optimal $\|\Delta \|_{\rm F}$.
With this in mind, we present asymptotic analysis, and then
construct a consistent estimator .

\subsection{Asymptotic analysis of matrices in the optimization method}
For the asymptotic analysis, similarly to ${K}(\vec{x}_{1},\ldots ,\vec{x}_{n})$,
let ${K}(\vec{s}_{1},\ldots ,\vec{s}_{n})$ be divided into
\begin{align}
&{K}(\vec{s}_{1},\ldots ,\vec{s}_{n})=
\begin{bmatrix}
   {K}_{11}(\vec{s}_{1},\ldots ,\vec{s}_{n}) & {K}_{12}(\vec{s}_{1},\ldots ,\vec{s}_{n})  \\
   {K}_{21}(\vec{s}_{1},\ldots ,\vec{s}_{n}) & {K}_{22}(\vec{s}_{1},\ldots ,\vec{s}_{n})
\end{bmatrix}, \\
&{K}_{11}(\vec{s}_{1},\ldots ,\vec{s}_{n}) \in \R^{\ell \times \ell},
\quad {K}_{12}(\vec{s}_{1},\ldots ,\vec{s}_{n}) \in \R^{\ell \times m},
\quad {K}_{21}(\vec{s}_{1},\ldots ,\vec{s}_{n}) \in \R^{m \times \ell},
\quad {K}_{22}(\vec{s}_{1},\ldots ,\vec{s}_{n}) \in \R^{m \times m}.
\end{align}
From the noise model as in~\eqref{eq:xsxil}, we see
\begin{align*}
{K}_{11}(\vec{x}_{1},\ldots ,\vec{x}_{n})={K}_{11}(\vec{s}_{1},\ldots ,\vec{s}_{n}).
\end{align*}
Under Assumption~\ref{as:s}, $\lim_{d\to \infty}{K}(\vec{s}_{1},\ldots ,\vec{s}_{n})$ exists, and let
\begin{align}\label{eq:kslim}
\lim_{d\to \infty}{K}(\vec{s}_{1},\ldots ,\vec{s}_{n})=
{K}^{(\infty)}(\vec{s}_{1},\ldots ,\vec{s}_{n})=
\begin{bmatrix}
   {K}_{11}^{(\infty)}(\vec{s}_{1},\ldots ,\vec{s}_{n}) & {K}_{12}^{(\infty)}(\vec{s}_{1},\ldots ,\vec{s}_{n})  \\
   {K}_{21}^{(\infty)}(\vec{s}_{1},\ldots ,\vec{s}_{n}) & {K}_{22}^{(\infty)}(\vec{s}_{1},\ldots ,\vec{s}_{n})
\end{bmatrix}.
\end{align}

The asymptotic structure of the kernel matrix ${K}(\vec{x}_{1},\ldots ,\vec{x}_{n})$ in
the next lemma is important and fundamental in the following asymptotic analysis.
\begin{lemma}\label{lem:keylem2}
Under Assumptions~\ref{as:s} and \ref{as:c}, 
$\lim_{d\to \infty}K(\vec{s}_{1},\ldots ,\vec{s}_{n})$
exists.
In addition, under Assumption~\ref{as:xil}, using the definition in \eqref{eq:kslim}, we have
\begin{align}\label{eq:key2}
\lim_{d\to \infty}{K}(\vec{x}_{1},\ldots ,\vec{x}_{n})=
\begin{bmatrix}
   {K}_{11}^{(\infty)}(\vec{s}_{1},\ldots ,\vec{s}_{n}) & \exp\left(-\gamma^{-1}\bar{\sigma}^{2}\right){K}_{12}^{(\infty)}(\vec{s}_{1},\ldots ,\vec{s}_{n})  \\
  \exp\left(-\gamma^{-1}\bar{\sigma}^{2}\right) {K}_{21}^{(\infty)}(\vec{s}_{1},\ldots ,\vec{s}_{n}) & \exp\left(-2\gamma^{-1}\bar{\sigma}^{2}\right)({K}_{22}^{(\infty)}(\vec{s}_{1},\ldots ,\vec{s}_{n})-I_{m})+I_{m}
\end{bmatrix}
\end{align}
with probability one.
\end{lemma}
\begin{proof}
It is clear that $\lim_{d\to \infty}{K}_{11}(\vec{x}_{1},\ldots ,\vec{x}_{n})={K}_{11}^{(\infty)}(\vec{s}_{1},\ldots ,\vec{s}_{n})$
from $\vec{x}_{i}=\vec{s}_{i}\ (i=1,\ldots ,\ell)$.
In addition, from Theorem~\ref{thm:main}, we see 
\begin{align*}
\lim_{d\to \infty}{K}_{22}(\vec{x}_{1},\ldots ,\vec{x}_{n})=
\exp\left(-2\gamma^{-1}\bar{\sigma}^{2}\right)({K}_{22}^{(\infty)}(\vec{s}_{1},\ldots ,\vec{s}_{n})-I_{m})+I_{m}
\quad \text{with probability one}.
\end{align*}

Similarly, the strong convergence of ${K}_{12}(\vec{x}_{1},\ldots ,\vec{x}_{n})$ is proved as follows.
For $1\le i \le \ell, \ell +1\le j \le n$, 
\begin{align*}
\exp\left(-\frac{\|\vec{x}_{i}-\vec{x}_{j}\|^2}{c_d}\right)
=
\exp\left(-\frac{\|\vec{s}_{i}-\vec{s}_{j}-\vec{\xi}_{j}\|^2}{c_d}\right)
=
\exp\left(-\frac{\|\vec{s}_{i}-\vec{s}_{j}\|^2-2(\vec{s}_{i}-\vec{s}_{j})^{\top}\vec{\xi}_{j}+\|\vec{\xi}_{j}\|^2}{c_d}\right).
\end{align*}
Note that $\lim_{d\to \infty}d^{-1}\|\vec{s}_{i}-\vec{s}_{j}\|^2$ exists under Assumption~\ref{as:s}.
Thus, under Assumptions~\ref{as:xi} and \ref{as:c}, we have
\begin{align*}
\lim_{d\to \infty}\exp\left(-\frac{\|\vec{x}_{i}-\vec{x}_{j}\|^2}{c_d}\right)
&=
\lim_{d\to \infty}\exp\left(-\frac{d}{c_d}\frac{\|\vec{s}_{i}-\vec{s}_{j}\|^2-2(\vec{s}_{i}-\vec{s}_{j})^{\top}\vec{\xi}_{j}+\|\vec{\xi}_{j}\|^2}{d}\right)
\\
&=
\exp\left(-\gamma^{-1}\bar{\sigma}^{2}\right)
\lim_{d\to \infty}\exp\left(-\frac{\|\vec{s}_{i}-\vec{s}_{j}\|^2}{d}\right).
\end{align*}
Therefore, $\lim_{d\to \infty}{K}_{12}(\vec{x}_{1},\ldots ,\vec{x}_{n})=\exp\left(-\gamma^{-1}\bar{\sigma}^{2}\right){K}_{12}^{(\infty)}(\vec{s}_{1},\ldots ,\vec{s}_{n})$ with probability one.
From the symmetry of the Gram matrix, $\lim_{d\to \infty}{K}_{21}(\vec{x}_{1},\ldots ,\vec{x}_{n})=\exp\left(-\gamma^{-1}\bar{\sigma}^{2}\right){K}_{21}^{(\infty)}(\vec{s}_{1},\ldots ,\vec{s}_{n})$ with probability one.
This completes the proof.
\end{proof}

To apply the optimization problem in~\eqref{eq:optpsmall}, we assume the following.

\begin{assumption}\label{as:k11}
${K}_{11}^{(\infty)}(\vec{s}_{1},\ldots ,\vec{s}_{n})$ in \eqref{eq:kslim} is nonsingular.
\end{assumption}

Assumption~\ref{as:k11} means that
the noiseless submatrix of
${K}(\vec{x}_{1},\ldots ,\vec{x}_{n})$, given by $\vec{s}_{1},\ldots ,\vec{s}_{\ell}$,
is nonsingular. 
The next theorem is crucial for constructing a consistent estimator, where the proof is due to easy calculations by~\eqref{eq:key2} in Lemma~\ref{lem:keylem2}.

\begin{theorem}\label{thm:main2}
Under Assumptions~\ref{as:s} and \ref{as:c}, 
$\lim_{d\to \infty}K(\vec{s}_{1},\ldots ,\vec{s}_{n})$
exists.
In addition, under Assumptions~\ref{as:xil} and \ref{as:k11}, using the definition in \eqref{eq:kslim}, we have
\begin{align}
& \lim_{d\to \infty}({K}_{22}(\vec{x}_{1},\ldots ,\vec{x}_{n})-
{{K}_{21}(\vec{x}_{1},\ldots ,\vec{x}_{n})}{{K}_{11}(\vec{x}_{1},\ldots ,\vec{x}_{n})}^{-1}{K}_{12}(\vec{x}_{1},\ldots ,\vec{x}_{n}))
\nonumber
\\
& =
\exp\left(-2\gamma^{-1}\bar{\sigma}^{2}\right)
({K}_{22}^{(\infty)}(\vec{s}_{1},\ldots ,\vec{s}_{n})-
{{K}_{21}^{(\infty)}(\vec{s}_{1},\ldots ,\vec{s}_{n})}{{K}_{11}^{(\infty)}(\vec{s}_{1},\ldots ,\vec{s}_{n})}^{-1}{K}_{12}^{(\infty)}(\vec{s}_{1},\ldots ,\vec{s}_{n})-I_{m})
+I_{m}
\end{align}
with probability one.
\end{theorem}
Importantly, the above theorem corresponds to Theorem~\ref{thm:main}.
With the above theorem in mind, let
\begin{align}\label{eq:tau}
\tau_{1}\le \dots \le \tau_{m}
\end{align}
denote the eigenvalues of
${K}_{22}(\vec{x}_{1},\ldots ,\vec{x}_{n})-
{{K}_{21}(\vec{x}_{1},\ldots ,\vec{x}_{n})}{{K}_{11}(\vec{x}_{1},\ldots ,\vec{x}_{n})}^{-1}{K}_{12}(\vec{x}_{1},\ldots ,\vec{x}_{n})$.
Similarly, let
\begin{align}\label{eq:nu}
\nu_{1}\le \dots \le \nu_{m}
\end{align}
denote the eigenvalues of
${K}_{22}(\vec{s}_{1},\ldots ,\vec{s}_{n})-
{{K}_{21}(\vec{s}_{1},\ldots ,\vec{s}_{n})}{{K}_{11}(\vec{s}_{1},\ldots ,\vec{s}_{n})}^{-1}{K}_{12}(\vec{s}_{1},\ldots ,\vec{s}_{n})$.

As a corollary to Theorem~\ref{thm:main2},
we see the following, which corresponds to Corollary~\ref{cor:eig}.

\begin{corollary}\label{cor:eig2}
Under the same assumptions as in Theorem~\ref{thm:main2}, we have
\begin{align}\label{eq:taulim}
\lim_{d\to \infty} \tau_{i}=\exp(-2\gamma ^{-1}\bar{\sigma}^{2})\lim_{d\to \infty} \nu_{i}-\exp(-2\gamma ^{-1}\bar{\sigma}^{2})+1  \quad (i=1,\ldots ,n)
\quad
\text{with probability one}.
\end{align}
The eigenvectors of
${K}_{22}(\vec{x}_{1},\ldots ,\vec{x}_{n})-
{{K}_{21}(\vec{x}_{1},\ldots ,\vec{x}_{n})}{{K}_{11}(\vec{x}_{1},\ldots ,\vec{x}_{n})}^{-1}{K}_{12}(\vec{x}_{1},\ldots ,\vec{x}_{n})$
are the consistent estimators of the eigenvectors of
${K}_{22}^{(\infty)}(\vec{s}_{1},\ldots ,\vec{s}_{n})-
{{K}_{21}^{(\infty)}(\vec{s}_{1},\ldots ,\vec{s}_{n})}{{K}_{11}^{(\infty)}(\vec{s}_{1},\ldots ,\vec{s}_{n})}^{-1}{K}_{12}^{(\infty)}(\vec{s}_{1},\ldots ,\vec{s}_{n})$.
\end{corollary}

\subsection{Proposed consistent estimator for rank deficient cases}\label{sec:newresults}
Finally, we construct a new estimator according to the discussion above.
More specifically, in the rank deficient case of ${K}^{(\infty)}(\vec{s}_{1},\ldots ,\vec{s}_{n})$ as in the previous section,
we propose a modified matrix
that has strong consistency for the noise model in~\eqref{eq:xsxil} in the asymptotic regime as $d\to \infty$.

To this end, assume that $K^{(\infty)}(\vec{s}_{1},\ldots ,\vec{s}_{n})$ is a low rank matrix as in Assumption~\ref{as:lowrank}.
Then, under Assumption~\ref{as:k11}, 
the smallest eigenvalue of ${K}_{22}(\vec{s}_{1},\ldots ,\vec{s}_{n})-
{{K}_{21}(\vec{s}_{1},\ldots ,\vec{s}_{n})}{{K}_{11}(\vec{s}_{1},\ldots ,\vec{s}_{n})}^{-1}{K}_{12}(\vec{s}_{1},\ldots ,\vec{s}_{n})$ has the limit value
\begin{align*}
\lim_{d\to \infty} \nu_{1}=0,
\end{align*}
where the eigenvalues are defined ascending order as in~\eqref{eq:nu}.
Then, from the definition in \eqref{eq:tau}, we have
\begin{align}\label{eq:taulim0}
\lim_{d\to \infty} \tau_{1}=
-\exp(-2\gamma ^{-1}\bar{\sigma}^{2})+1
\quad \text{with probability one}
\end{align}
due to the convergence feature \eqref{eq:taulim} in Coroallary~\ref{cor:eig2}.
Combined this with Lemma~\ref{lem:keylem2}, we propose a modified estimator
\begin{align}\label{eq:tildek2}
\widetilde{K}(\vec{x}_{1},\ldots ,\vec{x}_{n})=
\begin{bmatrix}
   {K}_{11}(\vec{x}_{1},\ldots ,\vec{x}_{n}) & (1-\tau_1)^{-1/2}{K}_{12}(\vec{x}_{1},\ldots ,\vec{x}_{n})  \\
   (1-\tau_1)^{-1/2}{K}_{21}(\vec{x}_{1},\ldots ,\vec{x}_{n}) & (1-\tau_1)^{-1}({K}_{22}(\vec{x}_{1},\ldots ,\vec{x}_{n})-I_{m})+I_{m}
\end{bmatrix}.
\end{align}

The procedure for calculating the proposed estimator $\widetilde{K}(\vec{x}_{1},\ldots ,\vec{x}_{n})$ is summarized as follows.

\vspace{10pt}

\noindent
[Proposed estimator]
\begin{itemize}
\item Compute the Gram matrix of the Gaussian kernel ${K}(\vec{x}_{1},\ldots ,\vec{x}_{n})$ as in \eqref{eq:gram}

\item Using the submatrices defined as in~\eqref{eq:ksubx} and \eqref{eq:ksubx4}, compute the smallest eigenvalue $\tau_{1}$ of ${K}_{22}(\vec{x}_{1},\ldots ,\vec{x}_{n})-
{{K}_{21}(\vec{x}_{1},\ldots ,\vec{x}_{n})}{{K}_{11}(\vec{x}_{1},\ldots ,\vec{x}_{n})}^{-1}{K}_{12}(\vec{x}_{1},\ldots ,\vec{x}_{n})$

\item Compute $\widetilde{K}(\vec{x}_{1},\ldots ,\vec{x}_{n})$ in \eqref{eq:tildek2} as the estimator
\end{itemize}

The asymptotic analysis above straightforwardly leads to
the next theorem that states strong consistency of the proposed estimator $\widetilde{K}(\vec{x}_{1},\ldots ,\vec{x}_{n})$.

\begin{theorem}\label{thm:main3}
Under Assumptions~\ref{as:s} and \ref{as:c}, 
$\lim_{d\to \infty}K(\vec{s}_{1},\ldots ,\vec{s}_{n})$
exists.
In addition, under Assumptions~\ref{as:lowrank}, \ref{as:xil}, and \ref{as:k11}, using the definition in \eqref{eq:kslim}, we have 
\begin{align}\label{eq:widetildeklim2}
\lim_{d\to \infty} \widetilde{K}(\vec{x}_{1},\ldots ,\vec{x}_{n})
=
\lim_{d\to \infty} K(\vec{s}_{1},\ldots ,\vec{s}_{n})
\quad \text{with probability one},
\end{align}
where $\widetilde{K}(\vec{x}_{1},\ldots ,\vec{x}_{n})$ is defined as in~\eqref{eq:tildek2}.
\end{theorem}
\begin{proof}
Noting \eqref{eq:key2} in Lemma~\ref{lem:keylem2}, we obtain~\eqref{eq:widetildeklim2}
from \eqref{eq:taulim0} and \eqref{eq:tildek2}.
\end{proof}

\begin{remark}
As described in Remark~\ref{rem:rankdeficient},
rank deficient cases are important.
One might notice that, if $K_{22}^{(\infty)}(\vec{s}_{1},\ldots ,\vec{s}_{n})$ is low rank,
then $\exp(-2\gamma ^{-1}\bar{\sigma}^{2})$ can be estimated by
the smallest eigenvalue of $K_{22}(\vec{s}_{1},\ldots ,\vec{s}_{n})$, as in the previous section.
The estimator proposed in this section is efficient for
the full rank case of $K_{22}^{(\infty)}(\vec{s}_{1},\ldots ,\vec{s}_{n})$.
For a finite $n$ and to preserve the full rank case of $K_{22}^{(\infty)}(\vec{s}_{1},\ldots ,\vec{s}_{n})$, 
the rank deficient case occurs when
$\lim_{d\to \infty}d^{-1}\|\vec{s}_{i}-\vec{s}_{j}\|^2=0$ for $1\le i \le \ell$ and $\ell+1\le j \le n$.
In this case, the $i^{\rm th}$ and $j^{\rm th}$ column vectors of $K(\vec{x}_{1},\ldots ,\vec{x}_{n})$ 
tend to become parallel as $d\to \infty$,
and thus $\exp(-2\gamma ^{-1}\bar{\sigma}^{2})$ can be estimated by
the ratio of their magnitudes.
In this section, however, we present a general estimation method
to extend to more general situations in the future, as discussed in Remark~\ref{rem:rankdeficient}.
\end{remark}

\begin{remark}
Some may wonder why we apply the above optimization problem~\eqref{eq:optp} 
to the noise model~\eqref{eq:xsxil} under consideration, thus we explain that point below.

This idea stems from the asymptotic analysis of the linear kernel.
Let $K_{\rm linear}(\vec{x}_{1},\ldots ,\vec{x}_{n})$ denote the Gram matrix given by the linear kernel
such that
\begin{align}
\left[ K_{\rm linear}(\vec{x}_{1},\ldots ,\vec{x}_{n}) \right]_{ij}={\vec{x}_{i}}^{\top}\vec{x}_{j}\qquad (1\le i,j \le n).
\end{align}
For the normalization, we consider $\lim_{d\to \infty}d^{-1}K_{\rm linear}(\vec{x}_{1},\ldots ,\vec{x}_{n})$
that has bounded elements under Assumptions~\ref{as:s} and \ref{as:xil}.
In addition, note that $\lim_{d\to \infty}d^{-1}K_{\rm linear}(\vec{s}_{1},\ldots ,\vec{s}_{n})$ exists.
Thus, let $K_{\rm linear}(\vec{s}_{1},\ldots ,\vec{s}_{n})$ be divided into
\begin{align*}
&{K}_{\rm linear}(\vec{s}_{1},\ldots ,\vec{s}_{n})=
\begin{bmatrix}
   {K}_{{\rm linear},11}(\vec{s}_{1},\ldots ,\vec{s}_{n}) & {K}_{{\rm linear},12}(\vec{s}_{1},\ldots ,\vec{s}_{n})  \\
   {K}_{{\rm linear},21}(\vec{s}_{1},\ldots ,\vec{s}_{n}) & {K}_{{\rm linear},22}(\vec{s}_{1},\ldots ,\vec{s}_{n})
\end{bmatrix},
\\
&{K}_{{\rm linear},11}(\vec{s}_{1},\ldots ,\vec{s}_{n})\in \R^{\ell \times \ell},\
{K}_{{\rm linear},12}(\vec{s}_{1},\ldots ,\vec{s}_{n})\in \R^{\ell \times m},\
{K}_{{\rm linear},21}(\vec{s}_{1},\ldots ,\vec{s}_{n})\in \R^{m \times \ell},\ 
{K}_{{\rm linear},22}(\vec{s}_{1},\ldots ,\vec{s}_{n})\in \R^{m \times m}.
\end{align*}
Then, under Assumptions~\ref{as:s} and ~\ref{as:xil}, we have
\begin{align*}
\lim_{d\to \infty}d^{-1}{K}_{\rm linear}(\vec{x}_{1},\ldots ,\vec{x}_{n})=
\begin{bmatrix}
   \lim_{d\to \infty}d^{-1}{K}_{{\rm linear},11}(\vec{s}_{1},\ldots ,\vec{s}_{n}) & \lim_{d\to \infty}d^{-1}{K}_{{\rm linear},12}(\vec{s}_{1},\ldots ,\vec{s}_{n})  \\
   \lim_{d\to \infty}d^{-1}{K}_{{\rm linear},21}(\vec{s}_{1},\ldots ,\vec{s}_{n}) & \lim_{d\to \infty}d^{-1}{K}_{{\rm linear},22}(\vec{s}_{1},\ldots ,\vec{s}_{n})+\bar{\sigma}^{2}I_{m}
\end{bmatrix}
\end{align*}
with probability one.
This feature corresponds to the optimization problem~\eqref{eq:optp}
in the case of the linear kernel function.

At first glance, 
a naive formulation to the Gaussian kernel matrix as in~\eqref{eq:optp}
appears to be meaningless because 
Lemma~\ref{lem:keylem2} states that
${K}_{12}(\vec{x}_{1},\ldots ,\vec{x}_{n})$
and
${K}_{21}(\vec{x}_{1},\ldots ,\vec{x}_{n})$
are inconsistent different from
$d^{-1}{K}_{{\rm linear},12}(\vec{x}_{1},\ldots ,\vec{x}_{n})$
and
$d^{-1}{K}_{{\rm linear},21}(\vec{x}_{1},\ldots ,\vec{x}_{n})$.
Nevertheless, this idea leads to a modified consistent estimator,
focusing on Theorem~\ref{thm:main2}
obtained by the structure of the bias in Lemma~\ref{lem:keylem2}.
This technical discovery can be considered one of the contributions of this paper.
\end{remark}

\section{Conclusion}\label{sec:conclusion}
In this paper, we presented asymptotic analysis
of the Gaussian kernel matrix
as $d\to \infty$.
More specifically, the asymptotic structure 
of the Gaussian kernel matrix presented by Karoui~\cite{Karoui2010}
is combined with the studies of the low rank matrix approximations
with constraints, leading to 
a new estimator with strong consistency in rank deficient cases as in Section~\ref{sec:newresults}.
From a technical viewpoint, the novelty overcoming the difficulty
is described only in Section~\ref{sec:mainresult2},
where the analysis is restricted to the Gaussian kernel under the statistical model 
representing the partial noise as in~\eqref{eq:xsxil}.
However, to the best of the author's knowledge,
this is the first study to directly apply the above procedure of the constrained low rank approximations.
It is worth noting that it successfully eliminates bias by focusing on its matrix structure 
and explicitly constructing estimators with strong consistency.

\

\noindent
{\bf Acknowledgment}

\noindent
This study was supported by JSPS KAKENHI Grant No. JP22K03422.





\end{document}